\documentclass[11pt]{amsart}

\usepackage{parskip}
\usepackage{amssymb}

\usepackage{mathrsfs}
\usepackage{color}

\usepackage{amsmath}
\usepackage{amsthm}

\usepackage{tikz}

\newtheorem{theorem}{Theorem}[section]

\newtheorem*{algorithm*}{Algorithm}
\newtheorem{lemma}[theorem]{Lemma}
\newtheorem{proposition}[theorem]{Proposition}
\newtheorem{propositiondef}[theorem]{Proposition-Definition}
\newtheorem*{lemma*}{Lemma}
\newtheorem{corollary}[theorem]{Corollary}

\newtheorem*{maintheorem*}{Main Theorem}
\newtheorem*{npalgorithm*}{The Newton-Puiseux Algorithm}
\usepackage{turnstile}
\theoremstyle{definition}
\newtheorem{definition}[theorem]{Definition}
\newtheorem{example}[theorem]{Example}
\newtheorem{question}[theorem]{Question}

\theoremstyle{remark} 
\newtheorem{remark}[theorem]{Remark}
\newtheorem{notation}[theorem]{Notation}

\newcommand{\so}{\mathcal{O}}

\newcommand{\fQ}{\mathbb{Q}}

\newcommand{\fR}{\mathbb{R}}

\newcommand{\fC}{\mathbb{C}}

\newcommand{\rN}{\mathbb{N}}

\newcommand{\rZ}{\mathbb{Z}}

\usepackage{hyperref}
\hypersetup{
    colorlinks=true,
    linkcolor=blue,
    citecolor=red,
    filecolor=magenta,      
    urlcolor=cyan,
}
 
\begin{document}

\title[ multiplier ideals of analytically irreducible plane curves]{ multiplier ideals of analytically irreducible plane curves}

\author{Mingyi Zhang}
\address{Department of Mathematics, Northwestern University, 2033 Sheridan Road, Evanston, IL 60208, USA}
\email{mingyi@math.northwestern.edu}

\begin{abstract}
Let $S$ be a Puiseux series of the germ of an analytically irreducible plane curve $Z$.  
We provide a new perspective to construct a set of polynomials $F=\{F_1,\ldots, F_{g-1}\}$ associated to $S$, which is a special choice of maximal contact elements constructed in \cite{AAB17} and approximate roots defined in \cite{Dur18}, \cite{AM73a}, \cite{AM73b}. Using these polynomials as building blocks, we describe 
a set of  generators of multiplier ideals of 
the form $\mathfrak{I}(\alpha Z)$ with  $0<\alpha<1$ a rational number, which recovers the results about irreducible plane curves in \cite{AAB17}, \cite{Dur18}.
\end{abstract}
\maketitle
 \section{Introduction}
 Let $X$ be a smooth complex variety and $D$ be an effective $\fQ$-divisor on $X$. The multiplier ideal associated to $D$ is defined as 
\begin{align}\label{E:multiideal}
    \mathfrak{I}(D)=\varphi_*\so_Y(K_{Y/X}-[\varphi^*D])
\end{align}
 where $\varphi:Y\to X$ is a log resolution of $(X,D)$ (see \cite[Definition 9.2.1]{Laz04}).
Multiplier ideals and their vanishing theorems are very useful in many areas, but an explicit formula for multiplier ideals is hard to give 
except in several special cases, see for example \cite{Bli04}, \cite{How11}, \cite{Mus06}, \cite{Tei07}, \cite{Tei08}, \cite{Tho14}, \cite{Tho16}. 
On the other hand, singularities of plane curves have been studied for a long time and there is a main way to describe them by 
using infinitely near points, Puiseux series, and invariants such as
characteristic exponents, multiplicity sequences, intersection multiplicities, etc.  
In this case, there are many results that describe partial information associated to multiplier ideals, like the log-canonical threshold 
or more generally jumping numbers, by analyzing characteristic exponents or the contribution of exceptional divisors of a resolution of singularity,
see \cite{AAD16}, \cite{GHM16}, \cite{HJ18}, \cite{Igu77}, \cite{Jar07}, \cite{Jar06}, \cite{Kuw99}, \cite{Nai09},  \cite{ST06}, \cite{Tuc10}, \cite{Sai00}. By analyzing  properties and invariants of infinitely near points, such as proximity, 
free (or satellite) points and multiplicity sequences, and using unique factorization theorems for complete primary valuation ideals in a regular local ring of dimension 2, 
J\"{a}rvilehto \cite{Jar06} provided a thorough description of the jumping numbers of multiplier ideals of a simple complete ideal, 
in particular, of an analytically irreducible plane curve. In \cite{HJ18}, the authors generalize the formula to any complete ideal in a regular local ring of dimension 2. Taking a different path, Naie \cite{Nai09} developed independently a complete description of the jumping numbers for analytically irreducible plane curves in terms of the Zariski exponents.
 From a different perspective, Tucker \cite{Tuc10} presented an algorithm 
to find jumping numbers for any plane curve. In \cite{AAD16}, the authors improved Tucker's algorithm to compute more efficiently the jumping numbers of any ideal $\mathfrak{a}$ in a two-dimensional local ring $\so_{X,O}$ with a rational singularity. Moreover, given a fixed log resolution $\varphi:X'\to X$ of $\mathfrak{a}$ and any jumping number $\lambda$ of $\mathfrak{a}$, this algorithm managed to compute an antinef divisor $D_{\lambda,\varphi}$ such that
\[\mathfrak{I}(\mathfrak{a}^{\lambda})=\varphi_*\so_{X'}(-D_{\lambda,\varphi})\]
where $\mathfrak{I}(\mathfrak{a}^{\lambda})$ is the multiplier ideal associated to the ideal $\mathfrak{a}$ and the coefficient $\lambda$ (see \cite[Definition 9.2.3]{Laz04}).

In this paper, we provide a method to choose a set of standard factors $F=\{F_1,F_2,\ldots, F_{g-1}\}$ (see Definition \ref{D:standardfactors}) for a germ of irreducible plane curve $Z$. Moreover, we construct the standard form for the equation of $Z$ in terms of these standard factors  
  (see Proposition \ref{P:2}). Combining a classical result (see Proposition \ref{P:3}) that describes how a partial sum of the Puiseux series $S$ of $Z$ determines the standard resolution, we show that 
the set 
\[\big{\{}x^{p_x}y^{p_0}F_1^{p_1}\cdots F_{g-1}^{p_{g-1}}~|~\rho(x^{p_x}y^{p_0}F_1^{p_1}\cdots F_{g-1}^{p_{g-1}})> \alpha\big{\}}\] 
is a set of  generators of $\mathfrak{I}(\alpha Z)$.
Consequently, we give a full description of multiplier ideals associated to the irreducible plane curve singularity $Z$ as follows. 
\begin{maintheorem*}
    There exists a formula for the multiplier ideals $\mathfrak{I}(\alpha Z)$ with $0<\alpha<1$ in terms of a set of standard factors $F_1,\ldots, F_{g-1}$.
    For the precise statement, see Theorem \ref{T:main2}.
\end{maintheorem*}

By describing multiplier ideals completely, 
we also derive a formula for 
all the jumping numbers of $Z$ in Corollary \ref{C:jumpingnumber}, which recovers \cite[Theorem 9.4]{Jar06}. 
 A complete calculation using the method in this paper is given in Example \ref{Ex:1}.

We have recently learned of the papers \cite{AAB17} and \cite{Dur18}. In \cite{AAB17}, the authors provided a method to explicitly compute any complete ideal in a smooth complex surface. In particular, their algorithm gives an explicit description of multiplier ideals of irreducible plane curves in terms of maximal contact elements (see \cite[\S 2.4]{AAB17}). In the thesis \cite{Dur18}, the author developed another method to get the same formula by an algorithm inspired by \cite{Tuc10a} and \cite{How11}. Our construction of standard factors coincides with a special choice of maximal contact elements of $Z$ defined in \cite[\S 2.4]{AAB17} and approximate roots defined in \cite[Chapter 2 \S 1]{Dur18} \cite[Definition 3.2]{VD19}. These irreducible polynomials have been defined in the literature classically, see for example \cite{AM73a} \cite{AM73b}.  The formula for the generators of multiplier ideals of irreducible plane curves is the same as in the papers above, but we use a different method to show these elements are enough to generate the ideal, which is hopefully of independent interest. While the algorithms in \cite{AAD16} and\cite{AAB17} are based on Lipman's unique factorization theorem (see \cite[Corollary 3.1]{Lip94}), our method uses a computational result of J\"{a}rvilehto (see \cite[Proposition 7.14]{Jar06}) and an algorithm to construct of a standard form of the equation of the germ.

\section*{Acknowledgement}
I would like to express my deepest appreciation to my advisor Mihnea Popa for the continuous support of this project. Special thanks to Manuel Gonz\'{a}lez Villa  for pointing out to me the thesis of Duran \cite{Dur18}, and a referee for pointing out the paper \cite{AAB17}.
I would like to thank Kevin Tucker, Lawrence Ein, Antoni Rangachev, Pedro Daniel Gonz\'{a}lez P\'{e}rez, Miguel Robredo,
Sebasti\'{a}n Olano, Yajnaseni Dutta, Juanyong Wang, Stephen Shing-Toung Yau and Huaiqing Zuo for useful conversations.

\section{Preliminary}
\subsection{Puiseux series}\label{S:sectionPui}
Denote by $\fC\langle\langle x\rangle\rangle$ the field of formal Laurent series
\[\sum_{i\ge r}c_i x^{\frac{i}{n}}\]
with $r,n\in\rZ$, $n\ge 1$ (see \cite[\S 1.2]{Cas00} for the construction of this field). 
\begin{definition}\cite[\S 1.2, \S 2.2]{Cas00}\label{D:Puiseux}
    \begin{itemize}
        \item For any \[S=\sum_{i\ge r}c_i x^{\frac{i}{n}}\in\fC\langle\langle x\rangle\rangle,\]
        we define the \emph{order} in $x$ of $S$ to be
        \begin{align*}
        o_x(S)=\begin{cases}
            \infty,&\textrm{if}~S=0;\\
            \frac{\min\{i|c_i\neq 0\}}{n},&~\textrm{otherwise}.\\
        \end{cases}
        \end{align*}
        \item \emph{Puiseux series} are all such series $S$ with $o_x(S)>0$. 
        \item For any Puiseux series $S$, we can write $S=\sum_{i>0}c_i x^{\frac{i}{n}}$ such that $n$ is coprime to $\textrm{gcd}\{i|c_i\neq 0\}$. 
        Then $n$ is called the \emph{polydromy order} of $S$ and denoted as $n=\nu(S)$.
        \item For any Puiseux series $S=\sum_{i>0}c_ix^{\frac{i}{n}}$, let 
        \[[S]_{<l}:=\sum_{0<\frac{i}{n}<l}c_i x^{\frac{i}{n}}\]
        and
        \[[S]_{\le l}:=\sum_{0<\frac{i}{n}\le l}c_i x^{\frac{i}{n}}.\]
       
        \item For $n=\nu(S)$ and each $n$-th root of unity $\epsilon$, we call the series 
        \[\sigma_\epsilon(S)=\sum_{i\ge r}\epsilon^ic_ix^{\frac{i}{n}}\]
        a \emph{conjugate} of $S$.
        Let $f\in\fC[[x,y]]$. We say a Puiseux series $S$ is a \emph{$y$-root} of $f$ if $f(x,S)=0$. Its conjugates are also Puiseux series. If $S$ is a $y$-root 
        of $f$, then all conjugates of $S$ are $y$-roots of $f$, too.
        The set of all conjugates of $s$ will be called the \emph{conjugacy class} of $S$. 
        We set \[f_S=\prod_{i=1}^{\nu(S)}(y-S^{i})\in \fC[[x]][y],\]
        where $S=S^1,\ldots,S^{\nu(S)}$ are conjugates in the conjugacy class of $S$. We know $f_S\in\fC[[x]][y]$ since all its coefficients are 
        invariant by conjugation. 
        \item Let $O$ be the origin of $\fC^2$ and fix local coordinates $x,y$ at $O$. 
        Let $f$ be the equation of the germ of a plane curve $Z$ at $O$, that is to say, $Z=(f=0)$, $f\in\fC\{x,y\}$ and $f(0,0)=0$. 
        We call the $y$-roots of $f$ the \emph{Puiseux series of the germ $Z$}.
        \item  We say a Puiseux series $S$ is \emph{modified} if $o_x(S)$ is not an integer. 
    \end{itemize}
\end{definition}
The following lemmas illustrates the relation between irreducible germs of plane curves and conjugacy classes of Puiseux series.
\begin{lemma}\cite[Corollary 2.2.4]{Cas00}\label{L:singleconj}
  Let $Z$ be the germ of a plane curve at $O$ and fix local coordinates $x,y$ such that $Z$ does not contain the germ of the $y$-axis.
Then $Z$ is irreducible if and only if all its Puiseux series are in a single conjugacy class.
\end{lemma}
\begin{lemma}\cite[Corollary 1.8.5]{Cas00}\label{L:gap}
    Let $f\in\fC\{x,y\}$ with no factor $x$. Then $f$ is irreducible if and only if $f=uf_S$, with $u\in\fC\{x,y\}$ a unit and $S$
    a convergent Puiseux series.
\end{lemma}
\subsection{The Newton-Puiseux algorithm}\label{S:NPalgorithm}

\begin{definition}Fix  a system of orthogonal coordinates $\alpha, \beta$ of the plane $\fR^2$.
 For any element \[f=\sum_{\alpha,\beta\ge0}c_{\alpha,\beta}x^{\alpha}y^{\beta}\in \fC[[x,y]],\]
  we denote by 
        \[\Delta(f)=\{(\alpha,\beta)~|~c_{\alpha,\beta}\neq 0\}\]
        a discrete set of points and call it the \emph{Newton diagram of} $f$. We consider the convex hull $\bar{\Delta}(f)$ of $\Delta(f)+(\fR^+)^2$ and call the union of 
        compact faces of $\bar{\Delta}(f)$ the \emph{Newton polygon} of $f$, denoted by $N(f)$. Notice that $N(f)$ may be a single vertex. Suppose the vertices of $N(f)$ 
       are $P_i=(\alpha_i,\beta_i)$, $i=0,\ldots,k$ with $\alpha_{i-1}<\alpha_i$ and $\beta_{i-1}>\beta_i$, $i=1,\ldots, k$. Then we define the height of $N(f)$ to be
       $h(N(f)):=\beta_0$.

\end{definition}
We will review the Newton-Puiseux algorithm which provides all $y$-roots of a given formal power series $f\in\fC[[x,y]]$.
Details of algorithm can be consulted in textbooks about singularities of plane curve (see for example 
\cite[\S 1.4]{Cas00}).
\begin{npalgorithm*}
    Fix a series
\[f(x,y)=\sum_{\alpha,\beta\ge 0}c^{(0)}_{\alpha,\beta}x^\alpha y^\beta\in\fC[[x,y]].\]
Assume further $h(N(f))>0$, then the Newton polygon $N(f)$ ends in the $\alpha$-axis. 

\noindent\textbf{Step 1:} If $N(f)$ ends above the $\alpha$-axis. Then we get a $y$-root $S=0$ and the algorithm stops here. 
    Otherwise, if $N(f)$ ends in the $\alpha$-axis, then we choose a side $\Gamma_0$ of $N(f)$. Set $\beta '=\min\{\beta|(\alpha,\beta)\in\Gamma_0\}$ and denote by $(\alpha ',\beta ')$ the
    corresponding point on $\Gamma_0$. Then we associate a polynomial 
    \[F_{\Gamma_0}=\sum_{(\alpha,\beta\in\Gamma)}c^{(0)}_{\alpha,\beta}Z^{\beta-\beta '}\in\fC[Z]\]
    to the side $\Gamma_0$. Choose a root $a$ of $F_{\Gamma_0}$. Write down the equation of $\Gamma_0$ as $n\alpha+m\beta=k$, where $\mathrm{gcd}(n,m)=1$. 
    The coefficients $n,m$ and the root $a$ determine a coordinate change
    \begin{align*}
     \begin{cases}
         x=x_1^{n}\\
         y=x_1^{m}(a+y_1).\\
     \end{cases}   
    \end{align*}
    Then we set $f(x,y)=x_1^kf_1(x_1,y_1)$. Denote
     \[f_1(x_1,y_1)=\sum_{\alpha_1,\beta_1\ge 0}c^{(1)}_{\alpha_1,\beta_1}x_1^{\alpha_1}y_1^{\beta_1}.\]

\noindent\textbf{Step 2:} Inductively, for $i\ge 1$, if the Newton polygon $N(f_{i})$ ends in the $\alpha_1$-axis, then we choose a side $\Gamma_i$ of $N(f_i)$. Similar as above we associate a polynomial 
$F_{\Gamma_i}$ to the side $\Gamma_i$ and choose a root $a_i$ of $F_{\Gamma_i}$. The equation of $\Gamma_i$ will be $n_i\alpha+m_i\beta=k_i$ where $\mathrm{gcd}(n_i,m_i)=1$. 
The new variables $x_{i+1},y_{i+1}$ are given by the rules
\begin{align*}
    \begin{cases}
        x_i=x_{i+1}^{n_i}\\
        y_i=x_{i+1}^{m_i}(a_i+y_{i+1}).\\
    \end{cases}   
   \end{align*}
We denote by $f_i=x_{i+1}^{k_i}f_{i+1}$.    

The algorithm will stop whenever $N(f_i)$ ends above the $\alpha_i$-axis. In this case, we will get a $y$-root
\[S=x^{\frac{m}{n}}(a+x_1^{\frac{m_1}{n_1}}(a_1+\cdots+x_{i-1}^{\frac{m_{i-1}}{n_{i-1}}}(a_{i-1}+0)\cdots)).\]
Otherwise, the algorithm may keep going on and we can only write down the initial part of the $y$-root as
\[S=ax^{\frac{m}{n}}+a_1x^{\frac{m_1}{nn_1}}+\cdots+a_{i}x^{\frac{m_i}{nn_1\cdots n_i}}+\cdots.\]
\end{npalgorithm*}

The following lemma will be useful for finding the standard form of the power series of an irreducible plane curve.
\begin{lemma}\cite[Proposition 1.5.7]{Cas00}\label{L:Puial}
    A Puiseux series $S$ is a $y$-root of $f$ if and only if it is obtained from $f$ by the Newton-Puiseux algorithm.
\end{lemma}
\subsection{The characteristic sequence and the multiplicity sequence}~

We recall further invariants associated to Puiseux series.
\begin{definition} \label{D:chaseq}
    \begin{itemize}
        \item For any Puiseux series 
        $S=\sum_{j\ge n}c_jx^{\frac{j}{n}}$
        that is not an integral power series and with polydromy order $\nu(S)=n$, it can be written as 
        \[S=\sum_{\substack{t\in\rN\\t<\frac{m_1}{n}}}c_{tn}x^t+a_1 x^{\frac{m_1}{n}}+\sum_{\substack{\eta\in(d_1)\\m_1<\eta<m_2}}b_\eta x^{\frac{\eta}{n}}+a_2x^{\frac{m_2}{n}}
            +\cdots+a_gx^{\frac{m_g}{n}}+\sum_{j>m_g}c_jx^{\frac{j}{n}},\]
        where  $d_0=n$, $d_i=\textrm{gcd}(m_i,\ldots,m_1,n)$, $m_i\notin (d_{i-1})$ and $d_{g-1}>d_g=1$. 
         Then the sequence 
        $(n;m_1,\ldots,m_g)$ is called the \emph{characteristic sequence} of $S$. Given any germ of irreducible plane curve singularity $Z$, 
        we define the \emph{characteristic sequence} of $Z$ to be the characteristic sequence of any $y$-root of $Z$. It is well defined since
        all $y$-roots of the irreducible curve $Z$ are conjugate by Lemma \ref{L:singleconj} and hence have the same characteristic sequence.   
         The characteristic sequence of $Z$ is independent of the choice of coordinates $x,y$ around $O$ as long as $Z$ is not tangent to the $y$-axis.
        \item For the germ of a plane curve $Z$, there exists a unique ``smallest'' resolution of the singularity which is a birational 
        morphism $\pi$
        consisted by the smallest 
        number of blow-ups 
        \[\pi:Y=Y_k\stackrel{\pi_k}{\xrightarrow{\hspace*{0.7cm}}}Y_{k-1}\stackrel{\pi_{k-1}}{\xrightarrow{\hspace*{0.7cm}}}\cdots \xrightarrow{\hspace*{0.7cm}} Y_1\stackrel{\pi_1}{\xrightarrow{\hspace*{0.7cm}}} Y_0=\fC^2\]
        with $\pi^*Z$ having normal crossing support. We call this resolution \emph{the standard resolution of $Z$} (see \cite[Definition, page 498]{BK86}). The multiplicity of the strict transform of $Z$ at each center $q_i$ of the blow-up $\pi_i$ is denoted as 
        $M_i$. Then the sequence $(M_1,\ldots,M_k)$ is called the \emph{multiplicity sequence} of $Z$.
    \end{itemize}
\end{definition}
\begin{remark}
    Fix local coordinates $x,y$ at $O$. Then the irreducible curve singularity $O\in Z$ is not tangent to $y$-axis if and only if the characteristic sequence of $Z$
    $(n;m_1,\ldots,m_g)$ satisfies that $n<m_1$. 
  
    In this paper, we will always assume $Z$ is not tangent to the $y$-axis, then the characteristic sequence is independent of the choice of coordinates $x,y$ around $O$.
  \end{remark}
 The following theorem of Enriques and Chisini \cite{EC24} tells us 
that the
multiplicity sequence of $Z$ and the characteristic sequence of $Z$ determine each other if $Z$ is an irreducible plane curve singularity. 
\begin{theorem}[Enriques-Chisini Theorem]\label{T:ECT}
    Given $(n;m_1,\ldots,m_g)$ the characteristic sequence of an irreducible plane curve singularity $Z$, 
    we get the following chain of $g$  Euclidean algorithms:
\begin{align}\label{E:3}
    \begin{cases}
    m_1=h_{1,0}n+r_{1,1}\\
    n=h_{1,1}r_{1,1}+r_{1,2}\\
    \vdots\\
    r_{1,k_1-1}=h_{1,k_1}r_{1,k_1}\\
    m_2-m_1=h_{2,0}r_{1,k1}+r_{2,1}\\
    r_{1,k_1}=h_{2,1}r_{2,1}+r_{2,2}\\
    \vdots\\
    r_{2,k_2-1}=h_{2,k_2}r_{2,k+2}\\
    \vdots\\
    m_g-m_{g-1}=h_{g,0}r_{g-1,k_2}+r_{g,1}\\
    \vdots\\
    r_{g,k_g-1}=h_{g,k_g}r_{g,k_g},\\
    \end{cases}
\end{align}
where we set $n=r_{1,0}$, $r_{i.k_i}=r_{i+1,0}$, and $1\le r_{i,j}<r_{i,j-1}$ for any $1\le i\le g$ and $1\le j\le k_i$.
Then the multiplicity sequence $(M_1,\ldots, M_k)$ of $Z$ is equal to 
\[(\overbrace{n,\ldots,n}^{h_{1,0}},\overbrace{r_{1,1},\ldots,r_{1,1}}^{h_{1,1}},\ldots,\overbrace{r_{1,k_1},\ldots,r_{1,k_1}}^{h_{1,k_1}+h_{2,0}},
\ldots, \overbrace{r_{g,k_g},\ldots,r_{g,k_g}}^{h_{g,k_g}}).\]
Conversely, given the multiplicity sequence of an irreducible plane curve singularity $Z$, one can recover the characteristic sequence of $Z$ 
by the chain of Euclidean algorithms. 
\end{theorem}

\subsection{Proximity matrix and multiplier ideals}~

We recall the proximity relation of inifinitely near points and its connection to multiplier ideals.
\begin{definition}\cite[\S 3.3]{Cas00}
    Let $O$ be a point on a smooth surface $Y$. We call points on the exceptional divisor $E$ of 
    blowing up $O$ on $Y$ \emph{points in the first infinitesimal neighborhood of $O$ on $Y$}. Inductively, 
    for $i>0$, we define \emph{the points in the $i$-th 
    infinitesimal neighborhood of $O$ on $Y$} to be the points in the first infinitesimal neighborhood of some point in the $(i-1)$-th infinitesimal 
    neighborhood of $O$. We call a point \emph{a infinitely near point of $O$ on $Y$} if 
    it is in the $i$-th infinitesimal neighborhood of $O$ on $Y$ for some $i>0$.
\end{definition}
\begin{definition}\label{D:4}
    Let $p,q$ be points equal or infinitely near to $O$. The point $q$ is said to be \emph{proximate to} $p$ if and only if it belongs, as an ordinary or
    infinitely near point, to the exceptional divisor $E^p$ of blowing up the point $p$ and we denote it by $q\succ p$.
The \emph{proximity matrix} of a sequence of infinitely near points $q_1,\ldots,q_k$ of $O$ on $\fC^2$
reads
\[P:=(p_{i,j})_{k\times k},~\textrm{where}~p_{i,j}=\begin{cases}1,&\textrm{if}~i=j;\\
    -1,&\textrm{if}~q_i\succ q_j;\\
    0,&\textrm{otherwise}.
\end{cases}\]
\end{definition}  
 \begin{notation}\label{N:gamtau}
     Set $\gamma_0=0$ and for $1\le i\le g$, set integers $\gamma_i$, $\tau_{i-1}$ such that
     \[\gamma_i=\gamma_{i-1}+\sum_{j=0}^{k_i}h_{i,j},\hspace{10pt}\textrm{and}\hspace{10pt}\tau_{i-1}=\gamma_{i-1}+h_{i,0}+1，\]
     where $h_{i,j}$ are given in (\ref{E:3}).
 \end{notation}
 \begin{definition}\cite[\S 3.6]{Cas00}\label{D:6}
An infinitely near point $p$ of $O$ is called a \emph{free point} of $O$ if it is proximate to just one point equal or infinitely near to $O$. 
Otherwise, $p$ is called a \emph{satellite point} of $O$.
     \end{definition}
 \begin{remark}
    Let 
    \[\pi:Y=Y_k\stackrel{\pi_k}{\xrightarrow{\hspace*{0.7cm}}}Y_{k-1}\stackrel{\pi_{k-1}}{\xrightarrow{\hspace*{0.7cm}}}\cdots \xrightarrow{\hspace*{0.7cm}} Y_1\stackrel{\pi_1}{\xrightarrow{\hspace*{0.7cm}}} Y_0=\fC^2\]
    be the standard resolution of $Z$.
       Let $q_1,\ldots, q_k$ be the centers of the blow-ups $\pi_1,\ldots,\pi_k$ in the standard resolution of $Z$.
 Then $q_r$ is a free point of $O$ when $\gamma_j<r\le \tau_j$ for some $0\le j\le g-1$, and $q_r$ is a satellite point of $O$ when 
 $\tau_j<r\le \gamma_{j+1}$ for some $0\le j\le g-1$.
 Centers $q_{\gamma_1}<q_{\gamma_2}<\cdots<q_{\gamma_g}$ (or $q_{\tau_0}<q_{\tau_1}<\cdots<q_{\tau_{g-1}}$) 
 are exactly all the terminal satellite (or free) points of the point basis $(M_1,\ldots, M_k)$, respectively, defined in \cite[Definition 3.1]{Jar06}. 
 \end{remark}

 \begin{definition}\label{N:coeffexc}
   Let $Z$ be the germ of an irreducible plane curve and
   let $q_1,\ldots, q_k$ be the centers of the blow-ups $\pi_1,\ldots,\pi_k$ in the standard resolution of $Z$. Denote by $P$ the proximity matrix of 
    infinitely near points $q_1,\ldots, q_j$ of $O$.
   We call the matrix $P^{-1}=(x_{i,j})_{k\times k}$ the \emph{inverse proximity matrix} of $Z$ and denote by 
    \[X_i=[x_{i,1},\ldots,x_{i,k}]\]
     the $i$th
 row of $P^{-1}$.
 \end{definition}
 By elementary computations we get the following formulas:
\begin{lemma}\label{L:7} 
  We have, for any $1\le i\le g$,  
    \[X_{\gamma_i}=[\overbrace{\frac{n}{d_i},\ldots,\frac{n}{d_i}}^{h_{1,0}}, 
    \overbrace{\frac{r_{1,1}}{d_i},\ldots,\frac{r_{1,1}}{d_i}}^{h_{1,1}}
    ,\ldots,
    \overbrace{\frac{r_{j-1,k_{i-1}} }{d_i},\ldots,\frac{r_{i-1,k_{i-1}} }{d_i}}^{h_{i-1,k_{i-1}}+h_{i,0}},
    \ldots,\overbrace{\frac{r_{i,k_i}}{d_i},\ldots,\frac{r_{i,k_i}}{d_i}}^{h_{i,k_i}},0,\ldots,0],\]
    and for any $0\le j\le g-1$,
    \[X_{\tau_j}=[\overbrace{\frac{n}{d_j},\ldots,\frac{n}{d_j}}^{h_{1,0}},
    \ldots,
    \overbrace{\frac{r_{j-1,k_{j-1}}}{d_j},\ldots,\frac{r_{j-1,k_{j-1}}}{d_j}}^{h_{j-1,k_{j-1}}+h_{j,0}},
    \ldots,
    \overbrace{\frac{r_{j,k_j}}{d_j},\ldots,\frac{r_{j,k_j}}{d_j}}^{h_{j,k_j}},\overbrace{1,\ldots,1}^{h_{j+1,0}+1},0,\ldots,0].\]
\end{lemma}
\begin{notation}
Let $G\in\fC\{x,y\}$.  Denote by $C_G$ the divisor defined  by $G$ and let
\[\textrm{ord}(G):=[\textrm{mult}_{q_1}\tilde{C}_G,\ldots,\textrm{mult}_{q_k}\tilde{C}_G],\]
where we denote by $\tilde{C_G}$ the strict transform of $C$ in $Y_1,\ldots, Y_k$.
\end{notation}
The following lemma is probably well known to experts, but we shall give a proof for the convenience of readers.
\begin{lemma}\label{L:pullbackG}
   Let 
   \[\pi:Y=Y_k\stackrel{\pi_k}{\xrightarrow{\hspace*{0.7cm}}}
   Y_{k-1}\stackrel{\pi_{k-1}}{\xrightarrow{\hspace*{0.7cm}}}\cdots \xrightarrow{\hspace*{0.7cm}} 
   Y_1\stackrel{\pi_1}{\xrightarrow{\hspace*{0.7cm}}} Y_0=\fC^2\] 
   be the standard resolution of $Z$ and $P^{-1}$ be the inverse proximity matrix of $Z$.
    Then 
    \[\pi^*C_G=\tilde{C_G}+(\mathrm{ord}(G)\cdot X_1 )E_1+\cdots+(\mathrm{ord}(G)\cdot X_k)E_k.\]
\end{lemma}
\begin{proof}
    By Definition \ref{N:coeffexc}, the $i$th column of $P^{-1}$ is denoted by $[x_{1,i},\ldots,x_{k,i}]^T$.
We claim, for any $1\le i\le k-1$,
\begin{align}\label{E:compuexcpull}
    (\pi_k\circ\pi_{k-1}\cdots\cdot\pi_{i+1})^{*}E_i=x_{1,i}E_1+\cdots+x_{k,i}E_k.
\end{align}
For $i=k-1$, we need to show $\pi_k^*E_{k-1}=x_{1,k-1}E_1+\cdots+x_{k-1,k-1}E_{k-1}+x_{k,k-1}E_k$. We know the center $q_k$ is 
always proximate to $q_{k-1}$, so $\pi_k^*E_{k-1}=E_{k-1}+E_{k}$. On the other hand,
 by the definition of $P$ and $P^{-1}$ and by elementary calculations, we get $x_{1,k-1}=\cdots=x_{k-2,k-1}=0$ 
 and $x_{k-1,k-1}=x_{k,k-1}=1$. Hence we proved the base case $i=k-1$. Inductively, suppose for $j\le i\le k-1$, 
 we have 
 \begin{align}\label{E:induction}
    (\pi_k\circ\pi_{k-1}\cdots\cdot\pi_{i+1})^{*}E_i=x_{1,i}E_1+\cdots+x_{k,i}E_k.
 \end{align}
 Let $[p_{1,j-1},\ldots, p_{k,j-1}]^T$ be the $(j-1)$-th column of $P$. Since for $l>j-1$,
 \[q_{l,j-1}=\begin{cases}
    -1,&\textrm{if}~q_{l}\succ q_{j-1};\\
    0,&\textrm{otherwise},
\end{cases}\]
we know 
 \[\pi_l^*E_{j-1}=E_{j-1}-p_{l,j-1}E_j.\]
By the assumptions (\ref{E:induction}), we have 
\begin{align*}
    (\pi_k\circ\cdots\circ\pi_j)^*E_{j-1}=&E_{j-1}+(-x_{j,j}p_{j,j-1})E_j+(-x_{j+1,j}p_{j,j-1}-x_{j+1,j+1}p_{j+1,j-1})E_{j+1}\\
&+\cdots+(-x_{k,j}p_{j,j-1}-\cdots-x_{k,k}p_{k,j-1})E_k\\
=&x_{1,j-1}E_1+\cdots+x_{k,j-1}E_k,
\end{align*}
where the last equality is give by $PP^{-1}=I$.
By induction we proved the claim. Then by the definition of $\mathrm{ord}(G)$, we proved that 
\begin{align*}
    \pi^*C_G=&\tilde{C}_G+\sum_{i=1}^k\mathrm{mult}_{q_i}\tilde{C}_G(\pi_k\circ\pi_{k-1}\cdots\cdot\pi_{i+1})^{*}E_i\\
    =&\tilde{C_G}+(\mathrm{ord}(G)\cdot X_1 )E_1+\cdots+(\mathrm{ord}(G)\cdot X_k)E_k.
\end{align*}
\end{proof}    

Let the standard resolution $\pi$ be the log resolution $\varphi$ in (\ref{E:multiideal}).
Then the multiplier ideal of the pair $(\fC^2,cZ)$, for any $c\in\fQ_+$, is equal to
\[\mathfrak{I}(cZ)=\pi_*\so_Y(K_{Y/\fC^2}-[\pi^*(c Z)]).\]
Since the singularity of $Z$ is isolated at $O$, the ideal sheaf $\mathfrak{I}(\alpha Z)$ is trivial 
away from $O$, hence we are only interested in 
the stalk at $O$. By Lemma \ref{L:7}, for $0<\alpha<1$, the ideal can be written as 
\begin{align}\label{E:multipl}
    \mathfrak{I}(\alpha Z)=\Big{\{ }G\in\fC\{x,y\}~|~G^{(i)}+1+a^{(i)}>\alpha b^{(i)},~\forall~1\le i\le k\Big{\} }.
\end{align}

\begin{definition}\label{D:8}
    Given a germ of an irreducible plane curve singularity $Z$ and fix coordinates $x,y$ such that
     $Z$ is not tangent to $y$-axis. Let $\pi$ be the standard resolution of $Z$.
 For any $G\in\fC\{x,y\}$,  denote by $C_G$ the curve defined by $G$ and we set
 \[\pi^*C_G=\tilde{C_G}+G^{(1)}E_1+\cdots+G^{(k)}E_k,\]
 where $E_i$ is the exceptional divisor of the blow-up $\pi_i$ or its strict transform under any blow-ups $\pi_j$.
  Then we
 define a function $\rho:\fC\{x,y\}\to\fQ_+$ such that
\[\rho(G)=\min_{1\le i\le k}\Big{\{}\frac{G^{(i)}+a^{(i)}+1}{b^{(i)}}\Big{\}}.\]
\end{definition}
 The following lemma is a direct consequence of \cite[Proposition 7.14]{Jar06} and Lemma \ref{L:pullbackG}. 
 \begin{lemma}
     We can simplify the formula for $\rho$ as   
     \begin{align}\label{E:8}
         \rho(G)=\min_{1\le i\le g}\Big{\{}\frac{\mathrm{ord}(G)\cdot X_{\gamma_i}+a^{(\gamma_i)}+1}{b^{(\gamma_i)}}\Big{\}}.
     \end{align}
 \end{lemma}
\subsection{The position of points and standard factors}
Fix a germ of an irreducible plane curve $O\in Z$ and fix local coordinates $x,y$. 
The following result indicates necessary and sufficient conditions on the Puiseux series of an irreducible germ $Z'$
 for it to go through some of the centers of blow-ups in the standard 
resolution of $Z$.
 \begin{proposition}\cite[Propositions 5.7.1, 5.7.3, 5.7.5]{Cas00}\label{P:3}
    Fix a germ of an irreducible plane curve singularity $O\in Z$ and local coordinates $x,y$ at $O$ such that the Puiseux series of $Z$ are modified and $Z$ is not tangent to 
    the $y$-axis. Fix one of its Puiseux series $S$ and the characteristic sequence is $(n;m_1,\ldots,m_g)$.
    Let $d_i=\mathrm{gcd}(n,m_1,\ldots,m_i)$. Denote by $\pi$ the standard resolution of $Z$ and by $q_1,\ldots, q_k$ centers of blow-ups consisting of $\pi$. 
    Use the same notations as in Notation \ref{N:coeffexc}. Consider another irreducible germ $Z'$ and 
    denote by $q_1',q_2',\ldots,q_{k'}'$ its centers of blow-ups in its standard resolution. 

    \begin{itemize}
        \item[(a)] For a free point $q_j$ with $\gamma_i<j\le \tau_i$ for some $0\le i\le g-1$. $q_j'=q_j$ if and only if 
        $Z'$ has a Puiseux series $S'$ such that the partial sums
        \[[S']_{\le \frac{m_{i-1}+(j-\gamma_i-1)d_{i-1}}{n}}=[S]_{\le \frac{m_{i-1}+(j-\gamma_i-1)d_{i-1}}{n}}. \]
        Furthermore, there is a projective absolute coordinate in the first neighborhood of $q_j$ such that the satellite point (the point 
        on the $y$-axis if $q_j=O$) has coordinate $\infty$ and, for any $a\in \fC$, $Z'$ goes through the point of coordinate $a$ if and only if 
        \[S'=[S]_{\le \frac{m_{i-1}+(j-\gamma_i-1)d_{i-1}}{n}}+ax^{\frac{m_{i-1}+(j-\gamma_i)d_{i-1}}{n}}.\]
        \item[(b)] The center $q_{\gamma_i}'=q_{\gamma_i}$ for some $i$ and has $q_{\gamma_i+1}'$ being a free point of $Z'$
        if and only if $Z'$ has a Puiseux series $S'$ such that
        \[S'=[S]_{<\frac{m_i}{n}}+a x^{\frac{m_i}{n}}+\cdots\]
        for some $a\in \fC-\{0\}$. Moreover, one may choose an absolute projective coordinate $z$ in the first neighborhood of $q_{\gamma_i}$ 
         so that, for any $a\neq 0$, $Z'$ goes through the point of coordinate $z$ in the first neighborhood of $q_{\gamma_i}$
        if and only if $a^{\frac{d_{i-1}}{d_i}}=z$.
        \item[(c)] Fix a satellite point $q_j$ with $\tau_i<j\le \gamma_{i+1}$ for some $0\le i\le g-1$. More precisely, either there exists $1\le t<k_{i+1}$ such that
        \[j-\tau_i=h_{i+1,1}+\cdots+h_{i+1,t-1}+r\]
        where $2\le r\le h_{i+1,t}+1$ or $t=k_{i+1}$ and 
        $2\le r=j+h_{i+1,t}-\gamma_{i+1}\le h_{i+1,t}$.
        
        Then $q_{j}'=q_j$ if and only if $Z'$ has a Puiseux series $S'$ of the form
        \[S'=[S]_{<\frac{m_{i+1}}{n}}+bx^{\frac{m'}{n'}}+\cdots\]
        where $n'=\nu(S')=\frac{dn}{d_i}$ for some $d\in\rN$ is the polydromy order of $S'$ and $m'$ satisfies the following condition:
        
        If we write $\frac{m'}{d}$ as a continued fraction in the form
\[\frac{m'}{d}=\frac{m_i}{d_i}+h_{i+1,0}'+\cfrac{1}{h'_{i+1,1}+\cfrac{1}{\ddots+\frac{1}{h_{i+1,t'}'}}},\]
then either $h_{i+1,0}'=h_{i+1,0},\ldots, h_{i+1,t-1}'=h_{i+1,t-1}$, $h_{i+1,t}'\ge r-1$ if $t'>t$, or 
$h_{i+1,0}'=h_{i+1,0},\ldots, h_{i+1,t-1}'=h_{i+1,t-1}$, $h_{i+1,t}'\ge r$ if $t'=t$.

    \end{itemize}
\end{proposition}
\begin{corollary}
Let $Z_1$, $Z_2$ be two irreducible plane curve singularities and fix coordinate $x,y$ such that 
 $Z_1$ and $Z_2$ are not tangent to $y$-axis in coordinates $x,y$.
  Assume $Z_1$ and $Z_2$ have the same characteristic sequence $(n;m_1,\ldots,m_g)$. 
  Then $Z_1$ and $Z_2$ have the same standard resolution if and only if there are a $y$-root $S_1$ of $Z_1$ and a $y$-root $S_2$ of $Z_2$
with $[S_1]_{<\frac{m_g}{n}}=[S_2]_{<\frac{m_g}{n}}$. 
\end{corollary}
    \begin{lemma}\cite[Lemma 8.5]{Jar06}   \label{L:eucalg}
    If $a\le b$ are positive integers and $\textrm{gcd}(a,b)=1$, then for any positive integer $u$ there exists 
    positive integers $s$ and $t$ such that $sa+tb=ab+u$.
\end{lemma}

\begin{propositiondef}\label{D:standardfactors}
    Let $S$ be a modified Puiseux series of an irreducible germ $Z$  with the characteristic sequence $(n;m_1,\ldots,m_g)$.
    Denote by
    \[ d_i=\mathrm{gcd}(m_i,\ldots,m_1,n), ~\forall ~i=1,\ldots,g \]
   and $h_{i,0}$ is given in (\ref{E:3}).
    Then there exist a set of 
        polynomials 
        \[H=\{H_{i,j}~|~1\le i\le g-1, 0\le j\le h_{i+1,0}\}\]  
        with \[H_i:=\sum_{j=0}^{h_{i+1,0}}H_{i,j}\]
         such that, for any $1\le i\le g-1$, $0\le j\le h_{i+1,0}$, or $i=g,j=0$, the polynomial
        \[F_{i,j}:=\Big{(}\big{(}(y^{\frac{n}{d_1}}+H_1)^{\frac{d_1}{d_2}}+H_2\big{)}^{\frac{d_2}{d_3}}+
        \cdots+H_{i-1}\Big{)}^{\frac{d_{i-1}}{d_{i}}}+H_{i,0}+H_{i,1}+\cdots+H_{i,j}\]
        has a $y$-root $[S_i]_{\le \frac{m_i+jd_i}{n}}=[S]_{\le \frac{m_i+jd_i}{n}}$ and the polydromy order $\nu(S_i)=\frac{n}{d_i}$, and $H_{i,j}$ is a linear 
    combination of $x^\alpha y^\beta F_{1,h_{2,0}}^{\delta_1}\cdots F_{i-1,h_{i,0}}^{\delta_{i-1}}$ satisfying, for any $1\le l\le g$,
    \[\mathrm{ord}(x^\alpha y^\beta F_{1,h_{2,0}}^{\delta_1}\cdots F_{i-1,h_{i,0}}^{\delta_{i-1}})\cdot X_{\gamma_l}\ge X_{\gamma_i+j}\cdot X_{\gamma_l}.\]
    For $1\le i\le g-1$, denote by $F_i=F_{i,h_{i+1,0}}$. Then we call the set $F=\{F_1,\ldots, F_{g-1}\}$ 
    a \emph{set of standard factors} of $S$. 
\end{propositiondef}
\begin{remark}
This set of standard factors coincides with a special choice of a set of the maximal contact elements of the germ $Z$ (see \cite[\S 2.4]{AAB17}) and approximate roots of $f$ (see \cite[Chapter 2 \S 1]{Dur18}). These polynomials have been defined classically in the literature years ago (see for example \cite[\S 1]{AM73a}, \cite[\S 1]{AM73b}).  Following the construction in the proof below, we provide an algorithm to choose these standard factors, which may be more effective to use.

\end{remark}
\begin{proof}
 By Lemma \ref{L:eucalg}, there exist positive integers $s,t$  such that $s\frac{n}{d_1}+t\frac{m_1}{d_1}=\frac{nm_1}{d_1^2}+1$. 
  By running the Newton-Puiseux algorithm,
         we get \[y^{\frac{n}{d_1}}-a_1^{\frac{n}{d_1}}x^{\frac{m_1}{d_1}}-b_{m_1+d_1}\lambda_1 x^s(y/a_1)^t\]
         has a $y$-root with a partial sum \[a_1x^{\frac{m_1}{n}}+b_{m_1+d_1}x^{\frac{m_1+d_1}{n}}.\]
          Therefore
         \[H_{1,0}=-a_1^{\frac{n}{d_1}}x^{\frac{m_1}{d_1}}~\textrm{and}~H_{1,1}=-b_{m_1+d_1}\lambda_1 x^s(y/a_1)^t\]
          satisfied the desired conditions.
    Inductively, for $1\le q\le h_{2,0}$, suppose we have constructed 
    \[F_{1,q}=y^{\frac{n}{d_1}}-a_1^{\frac{n}{d_1}}x^{\frac{m_1}{d_1}}+H_{1,1}+\cdots+H_{1,q}\] which 
    has a $y$-root with a partial sum $[S]_{\le \frac{m_1+qd_1}{n}}$, and suppose that after taking coordinate changes 
       \[\begin{cases}
        x=x_{1}^{n/d_1}=x_2^{n/d_1}=\cdots=x_{q+1}^{n/d_1};\\
        y=x_1^{m_1/d_1}(y_1+a_1)=\cdots=x_{q+1}^{m_1/d_1}(x_{q+1}^qy_{q+1}+b_{m_1+qd_1}x_{q+1}^q+\cdots+b_{m_1+d_1}x_{q+1}+a_1),\\
    \end{cases} \]
     we can write
    \[y^{\frac{n}{d_1}}-a_1^{\frac{n}{d_1}}x^{\frac{m_1}{d_1}}+H_{1,1}+\cdots+H_{1,q}=x_{q+1}^{\frac{nm_1}{d_1^2}+
    q}(\lambda_1y_{q+1}+\sum_{\alpha+\beta>1}c^{(q)}_{\alpha,\beta}x_{q+1}^\alpha y_{q+1}^\beta).\]
       By Lemma \ref{L:eucalg} we can construct $H_{1,q+1}$, a linear combination of $x^{s}y^t$ with suitable coefficients such that 
      , after taking coordinate changes,
       \[\begin{cases}
           x_{q+1}=x_{q+2},\\
           y_{q+1}=x_{q+1}(y_{q+2}+b_{m_1+(q+1)d_1}),\\
       \end{cases}\]
we have $F_{1,q+1}=F_{1,q}+H_{1,q+1}$ is of the form 
\[F_{1,q+1}=x_{q+2}^{\frac{nm_1}{d_1^2}+q+1}(\lambda_1y_{q+2}+\sum_{\alpha+\beta>1}c_{\alpha,\beta}^{(q+1)}x_{q+2}^\alpha y_{q+2}^{\beta}).\]
By the Newton-Puiseux Algorithm and by Lemma \ref{L:singleconj}, we know that $F_{1,q+1}$ has a $y$-root with a partial sum
      \[a_1x^{\frac{m_1}{n}}+b_{m_1+d_1}x^{\frac{m_1+d_1}{n}}+\cdots+b_{m_1+(q+1)d_1}x^{\frac{m_1+(q+1)d_1}{n}}.\]
      By induction, we proved that there exists polynomials $H_{1,0},\ldots, H_{1,h_{2,0}}$ satisfying the desired conditions. Notice that 
      $\mathrm{ord}(x^sy^t)\cdot X_{\gamma_l}=\frac{n}{d_l}s+\frac{m_1}{d_l}t$, so we obtain that $H_{1,j}$ is a linear combination of $x^sy^t$
      satisfying, for any $1\le l\le g$, 
      \[\mathrm{ord}(x^sy^t)\cdot X_{\gamma_l}\ge \frac{d_1}{d_l}(\frac{nm_1}{d_1^2}+j)\ge X_{\gamma_1+j}\cdot X_{\gamma_l}.\]
     Now we want to show there exists $H_{2,0}$ 
        such that $F_1^{\frac{d_1}{d_{2}}}+H_{2,0}$ has a $y$-root with a partial sum
       \[a_1x^{\frac{m_1}{n}}+
            \sum_{\substack{\eta\in(d_1)\\m_1<\eta<m_2}}b_\eta x^{\frac{\eta}{n}}+a_2x^{\frac{m_2}{n}},\]
            and $H_{2,0}$ is a linear combination of $x^\alpha y^\beta F_1^{\delta_1}$ satisfying, 
\[\mathrm{ord}(x^\alpha y^\beta F_1^{\delta_1})\cdot X_{\gamma_l}\ge X_{\gamma_2}\cdot X_{\gamma_l}.\]
           By the previous arguments, we have 
           \[F_1=y^{\frac{n}{d_1}}+H_1=x_{h_{2,0}+1}^{\frac{nm_1}{d_1^2}+h_{2,0}}
           (\lambda_1y_{h_{2,0}+1}+\sum_{\alpha+\beta>1}c^{(h_{2,0})}_{\alpha,\beta}x_{h_{2,0}+1}^\alpha y_{h_{2,0}+1}^\beta).\] 
    By Lemma \ref{L:eucalg}, for any positive integer $w$, there exists positive integers $s_i, t_i$ such that 
    $s_i\frac{n}{d_1}+t_i\frac{m_1}{d_1}=\frac{nm_1}{d_1d_2}+h_{2,0}\frac{d_1}{d_2}+w$ (since $\frac{nm_1}{d_1d_2}>\frac{nm_1}{d_1^2}$).
    So, by setting $r=m_2-m_1-h_{2,0}d_1$, we can construct a linear combination $H_{2,0}^{(1)}$ of $x^sy^t$ such that 
    \[F_1^{\frac{d_1}{d_2}}+H_{2,0}^{(1)}=x_{h_{2,0}+1}^{\frac{nm_1+h_{2,0}d_1^2}{d_1d_2}}(\lambda_1^{\frac{d_1}{d_2}}y_{h_{2,0}+1}^\frac{d_1}
    {d_2}-a_2\lambda_1^{\frac{d_1}{d_2}}x_{h_{2,0}+1}^\frac{r}{d_2}+\sum_{\frac{d_1}{d_2}\alpha+\frac{r}{d_2}\beta>\frac{d_1r}{d_2^2}}
    c_{\alpha,\beta}^{(h_{2,0}+1)}x^\alpha y^\beta).\]
    For any positive integer $w$, by Lemma \ref{L:eucalg}, there exist positive integers $u,v$ such that 
    \[u\frac{m_2-m_1}{d_2}+v\frac{d_1}{d_2}=\frac{m_2-m_1}{d_2}\cdot\frac{d_1}{d_2}+t.\]
   Choose some $0<u<\frac{d_1}{d_2}$.
    Then 
    \begin{align}\label{E:Cor1}
    u(\frac{nm_1}{d_1d_2}+\frac{m_2-m_1}{d_2})+(\frac{d_1}{d_2}-u-1)\frac{nm_1}{d_1d_2}+\frac{nm_1}{d_1d_2}+v\frac{d_1}{d_2}=\frac{nm_1+(m_2-m_1)d_1}{d_2^2}+w.   
\end{align}
    By Lemma \ref{L:eucalg} again, there exist positive integers $s,t$ such that 
    \begin{align}\label{E:Cor2}
        s\frac{n}{d_2}+t\frac{m_1}{d_2}=\frac{nm_1}{d_1d_2}+v\frac{d_1}{d_2}.
    \end{align}
    Combining (\ref{E:Cor1}) and (\ref{E:Cor2}) we obtain
    \[u(\frac{nm_1}{d_1d_2}+\frac{m_2-m_1}{d_2})+\Big{[}(\frac{d_1}{d_2}-u-1)\frac{m_1}{d_1}+s\Big{]}\frac{n}{d_2}+t\frac{m_1}{d_2}=\frac{nm_1+(m_2-m_1)d_1}{d_2^2}+w.\]
    Set $s'=(\frac{d_1}{d_2}-u-1)\frac{m_1}{d_1}+s$, then we obtain positive integers $s',t,u$ such that 
    \[\frac{n}{d_2}s'+\frac{m_1}{d_2}t+\frac{nm_1+d_1(m_2-m_1)}{d_1d_2}u=\frac{nm_1+d_1(m_2-m_1)}{d_2^2}+w.\]
    Thus we can construct $H_{2,0}^{(2)}$,  a linear combination of $x^\alpha y^\beta F_1^{\delta_1}$ such that, 
    after coordinate changes,
    \[\begin{cases}
        x_{h_{2,0}+1}=x_{2,0}^{\frac{d_1}{d_2}},\\
        y_{h_{2,0}+1}=x_{2,0}^{\frac{m_2-m_1-h_{2,0}d_1}{d_2}}(y_{2,0}+a_2)
    \end{cases}\]
    wer have
    \[F_1^{\frac{d_1}{d_2}}+H_{2,0}^{(1)}+H_{2,0}^{(2)}=x_{2,0}^{\frac{nm_1+(m_2-m_1)d_1}{d_2^2}}(\lambda_2y_{2,0}+
    \sum_{\alpha+\beta>1}c^{(2,0)}_{\alpha,\beta}x_{2,0}^\alpha y_{2,0}^\beta).\]
    By elementary calculations, we get, for $1\le l\le g$, 
    \[\mathrm{ord}(x^\alpha y^\beta F_{1,h_{2,0}}^{\delta_1})\cdot X_{\gamma_l}\ge X_{\gamma_2}\cdot X_{\gamma_l}.\]
    Following very similar arguments as above, we showed that there exists $F_1,\ldots, F_{g-1}$ that satisfied the desired conditions.
\end{proof}
\begin{example}
    Let $Z=(y^6-6x^2y^5+9x^4y^4-2x^5y^3+6x^7y^2+x^{10}-9x^{11}=0)$ be the germ of a plane curve. After running the Newton-Puiseux Algorithm, we know that $Z$ is irreducible.
     The characteristic sequence 
    of $Z$ is $(n;m_1,m_2)=(6,10,13)$ and it has a $y$-root $S$ with 
    \[[S]_{\le\frac{13}{6}}=x^{\frac{5}{3}}+x^2+x^{\frac{13}{6}}.\]
    We may construct a set of standard factors of $S$ to be 
    \[F=\{F_1=y^3-x^5-3x^2y^2\}.\]
    The choice of standard factors is not unique. For example, we can choose $F_1=y^3-x^5-3x^2y^2+ax^sy^t$ with $3s+5t>16$ and $a\in\fC$.
\end{example}

To compute $\mathrm{ord}(x^{p_x}y^{p_0}F_1^{p_1}\cdots F_{g-1}^{p_{g-1}})\cdot X_{\gamma_l}$ and $\rho(x^{p_x}y^{p_0}F_1^{p_1}\cdots F_{g-1}^{p_{g-1}})$ efficiently, we shall use the following formulas:
\begin{lemma}\label{L:calculations}
    Denote by $(n;m_1,\ldots,m_g)$ the characteristic sequence of $Z$ and by $(M_1,\ldots,M_k)$ the multiplicity sequence of $Z$.
    Let $\pi:Y\to Y_0=\fC^2$ be the standard resolution of $Z$ and set 
        \[K_{Y/Y_0}=a^{(1)}E_1+\cdots+a^{(k)}E_k,~\pi^*Z=\tilde{Z}+b^{(1)}E_1+\cdots+b^{(k)}E_k.\]
Then, for any $1\le i\le g$ and $1\le j\le g-1$, we have 
\[a^{(\gamma_i)}+1=\frac{m_i+n}{d_i},~b^{(\gamma_i)}=\frac{M_1^2+\cdots+M_{\gamma_i}^2}{d_i},\] 
and for any set of standard factors $F=\{F_1,\ldots, F_{g-1}\}$ we have
\[\mathrm{ord}(F_j)\cdot X_{\gamma_i}=\begin{cases}
    \frac{b^{(\gamma_i)}}{d_j}=\frac{M_1^2+\cdots+M_{\gamma_i}^2}{d_id_j}&\textrm{when}~i\le j;\\
    \frac{b^{(\gamma_j)}}{d_i}+\frac{m_{j+1}-m_j}{d_i}=\frac{M_1^2+\cdots+M_{\gamma_{j+1}}^2}{d_id_j}&\textrm{when}~i>j.\\
\end{cases}
\]    
\end{lemma}
\begin{proof}
    Since $K_{Y_i}-\pi_i^*K_{Y_{i-1}}=E_i$, by (\ref{E:compuexcpull}), we get \[a^{(\gamma_i)}=[1,\ldots,1]\cdot X_{\gamma_i}.\] 
    By Lemma \ref{L:pullbackG}, we have 
    \[b^{(\gamma_i)}=[M_1,\ldots,M_k]\cdot X_{\gamma_i}.\]
    Since $F_i$ has a $y$-root $S_i$ with $[S_i]_{\le \frac{m_i+h_{i+1,0}d_i}{n}}=[S]_{\le \frac{m_i+h_{i+1,0}d_i}{n}}$ and the polydromy order is $\frac{n}{d_i}$, 
    by Proposition \ref{P:3}, we get $\mathrm{ord}(F_i)=X_{\tau_i}$. Then, using formulas in Lemma \ref{L:7} and (\ref{E:2}), we obtain the desired results.
\end{proof}
\subsection{Standard form of the power series of an irreducible plane curve}
When the initial term is fixed, i.e., for any Puiseux series starting with the term $a_1x^{\frac{m}{n}}$ with the polydromy order equal to $n$,
  the polynomial solved by this Puiseux series is, up to multiplicative constant, of the form
\[(y^{n'}-a_1^{n'}x^{m'})^{d}+\sum_{n\alpha+m\beta>nm}c_{\alpha,\beta} x^\alpha y^{\beta}\]
where $d=\textrm{gcd}(n,m)$ and $n'=\frac{n}{d}$, $m_1'=\frac{m}{d}$ (see for example \cite[Proposition 2.2.5]{Cas00}). The following 
proposition is a generalization of this result, giving a standard form when we fix a partial sum of a Puiseux series of $Z$.
 The idea of the proof is similar. 
    \begin{proposition}\label{P:2}
        Let $S$ be a modified Puiseux series of an irreducible germ $Z$  with the characteristic sequence $(n;m_1,\ldots,m_g)$.
        Let $F=\{F_1,\ldots, F_{g-1}\}$ be a set of standard factors of $S$ and let $F_{i,j}$ be the polynomials as defined in Proposition \ref{D:standardfactors}. Denote by
        \[ d_i=\mathrm{gcd}(m_i,\ldots,m_1,n), ~\forall ~i=1,\ldots,g. \]
        Then, up to a multiplicative unit factor in $\fC\{x,y\}$, for any $1\le i\le g-1$, $0\le j\le h_{i+1,0}$, or $i=g$, $j=0$,
        the power series $f$ of $Z$
        is of the form
        \[F_{i,j}^{d_{i}}+F^{(i,j)}_{ext},\]
        where  the power series
        $F^{(i,j)}_{ext}$ is a linear combination of  $x^{\alpha}y^{\beta}F_{1}^{\delta_1}F_{2}^{\delta_2}\cdots F_{g-1}^{\delta_{g-1} }$ satisfying, for $1\le l\le g$, 
         \[\mathrm{ord}(x^{\alpha}y^{\beta}F_{1}^{\delta_1}F_{2}^{\delta_2}\cdots F_{g-1}^{\delta_{g-1} })\cdot X_{\gamma_l}>\mathrm{ord}( F_{i,j}^{d_{i}})\cdot X_{\gamma_l}.\]
          
\end{proposition}
\begin{proof}
Let $f$ be a power series of $Z$. By Lemma \ref{L:Puial} the $y$-root $S$ is obtained from $f$ 
    by the Newton-Puiseux algorithm.
Since the term with the smallest fractional power is $a_1x^{\frac{m_1}{n}}$, we know from Step 0 of the Newton-Puiseux 
algorithm (see \S \ref{S:NPalgorithm}) that 
    $N(f)$ has a single side $\Gamma_0$ (because $Z$ is irreducible) and the slope of $\Gamma_0$ is $-\frac{n}{m_1}$. 
    By Lemma \ref{L:singleconj}, any root of $F_{\Gamma_0}$ is equal to $\epsilon^{m_1}a_1$ for some $\epsilon$ such that $\epsilon^n=1$, the first coefficient of one of 
    the $y$-root $\sigma_\epsilon(s)$. Since $d_1=\mathrm{gcd}(n,m_1)$, 
    we know all roots of 
    $F_{\Gamma_0}$ are of the form $\epsilon'a_1$ where $\epsilon'^{\frac{n}{d_1}}=1$.
   We know the degree of $F_{\Gamma_1}$ is $h(N(f))=\nu(s)=n$. This 
    implies that the equation of
    $\Gamma_0$ is 
    \[\frac{n}{d_1}\alpha+\frac{m_1}{d_1}\beta=\frac{nm_1}{d_1}\] and 
    \[F_{\Gamma_0}=(Z^{\frac{n}{d_1}}-a_1^{\frac{n}{d_1}})^{d_1}\] up to a multiplicative constant.
     From this equation 
    we know $f$ is of the form
   \begin{align}\label{E:1}
    (y^{\frac{n}{d_1}}-a^{\frac{n}{d_1}}x^{\frac{m_1}{d_1}})^{d_1}+\sum_{n\alpha+m_1\beta>m_1n}c^{(0)}_{\alpha,\beta}x^{\alpha}y^{\beta}.
   \end{align}
   Now we want use the similar argument on the next nonzero term in the $y$-root $S$ to get more details of the standard from of $f$.
Following the algorithm, the first coordinate change in the process of producing the $y$-root $s$ in the algorithm is given as
\begin{align}\label{E:2}
    \begin{cases}
        x=x_1^{\frac{n}{d_1}}\\
        y=x_1^{\frac{m_1}{d_1} }(y_1+a_1).\\
    \end{cases}
\end{align}
We get the form (\ref{E:1}) is equal to 
\[x_1^{\frac{nm_1}{d_1}}\Big{[}\big{(}(y_1+a_1)^{\frac{n}{d_1}}-a^{\frac{n}{d_1}}\big{)}^{d_1}
+\sum_{n\alpha+m_1\beta>m_1n}c_{\alpha,\beta}^{(0)}x_1^{\frac{n\alpha+m_1\beta-nm_1}{d_1}}(y_1+a_1)^{\beta}\Big{]}.\]
Following the Newton-Puiseux algorithm, we set $f_1=x_1^{-\frac{nm_1}{d_1}}f$. Then the degree of $f_1$ is $d_1$ and $h(N(f_1))=d_1$.
Then there are two cases. 

   \noindent\textbf{Case 1}: The next nonzero fractional power term is $b_{m_1+rd_1}x^{\frac{m_1+cd_1}{n}}$ 
    with the integer $1\le r\le h_{2,0}$. Then from Step 1 of the algorithm we know that $N(f_1)$ has a single side $\Gamma_1$ 
    and the slope of $\Gamma_1$ is $-\frac{d_1}{rd_1}=-\frac{1}{r}$. By Lemma \ref{L:Puial}, any root of $F_{\Gamma_1}$ is equal to 
    $\epsilon^{m_1+rd_1} b_{m_1+rd_1}$ where $\epsilon^n=1$ and $\epsilon^{m_1}=1$ (since the first coefficient is a fixed number $a_1$).
     Hence we get $b_{m_1+rd_1}$ is the only root of $F_{\Gamma_1}$ and the degree of $F_{\Gamma_1}$ is $d_1=h(N(f_1))$. Therefore 
    the equation of $\Gamma_1$ is 
    \[\alpha_1+r\beta_1=r,\]
    and 
     \[F_{\Gamma_1}=(Z-b_{m_1+rd_1})^{d_1}.\]
      Therefore the series $f_1$ of variables $x_1,y_1$, is of the form 
    \[ (y_1-b_{m_1+rd_1}x_1^r)^{d_1}+\sum_{\alpha_1+r\beta_1>rd_1}c_{\alpha_1,\beta_1}^{(1)}x_1^{\alpha_1}y_1^{\beta_1}.\]
We claim that $f$ is of the form
    \[(y^{\frac{n}{d_1}}-a^{\frac{n}{d_1}}x^{\frac{m_1}{d_1}}+H_{1,r})^{d_1}+
    \sum_{n\alpha+m_1\beta+\frac{m_1n+rd_1^2}{d_1}\delta_1>m_1n+rd_1^2}c^{(1)}_{\alpha,\beta,\delta_1}x^{\alpha}y^{\beta}(y^{\frac{n}{d_1}}-a^{\frac{n}{d_1}}x^{\frac{m_1}{d_1}})^{\delta_1}.\]
    where $y^{\frac{n}{d_1}}-a^{\frac{n}{d_1}}x^{\frac{m_1}{d_1}}+H_{1,r}$ has a $y$-root with initial terms $a_1x^{\frac{m_1}{n}}+b_{m_1+rd_1}x^{\frac{m_1+rd_1}{n}}$ and 
    with polydromy order $\frac{n}{d_1}$. Indeed, up to a multiplicative constant,
    \[D:=f-(y^{\frac{n}{d_1}}-a^{\frac{n}{d_1}}x^{\frac{m_1}{d_1}}+H_{1,r})^{d_1}=x_1^{\frac{m_1n}{d_1}}\cdot
    \sum_{\alpha_1+r\beta_1>rd_1}c_{\alpha_1,\beta_1}'x_1^{\alpha_1}y_1^{\beta_1}.\]
    Denote by $\tilde{N}$ the Newton polygon of the difference $D$, and pick a vertex $x_1^sy_1^t$ of $\tilde{N}$. 
By (\ref{E:2}) we observe that, the term $c'_{s,t}x_1^sy_1^t$ must be provided by polynomials, up to multiplicative constants, of the form
$x^\alpha y^\beta (y^{\frac{n}{d_1}}-a_1^{\frac{n}{d_1}}x^{\frac{m_1}{d_1}})^t$ 
such that \[\frac{n}{d_1}\alpha+\frac{m_1}{d_1}\beta+\frac{m_1n}{d_1^2}t=s.\]
 Combining with the condition $s-\frac{m_1n}{d_1}+rt>rd_1$, we obtain 
\[n\alpha+m_1\beta+\frac{m_1n+rd_1^2}{d_1}t>m_1n+rd_1^2.\]
By substracting a suitable linear combination of polynomials of such form from the difference $D$, all terms supported on $\tilde{N}$ 
are eliminated and we can do the similar argument for the new difference. This process will terminate in finite steps since there are finitely many $(\alpha,\beta)$ such that $n\alpha+m_1\beta\le m_1n+rd_1^2$. Thus the difference $D$ is of the form 
\[ \sum_{n\alpha+m_1\beta+\frac{m_1n+rd_1^2}{d_1}\delta_1>m_1n+rd_1^2}c^{(1)}_{\alpha,\beta,\delta_1}x^{\alpha}y^{\beta}(y^{\frac{n}{d_1}}-a^{\frac{n}{d_1}}x^{\frac{m_1}{d_1}})^{\delta_1}.\]

    \noindent\textbf{Case 2}: Suppose that the next nonzero term of $S$ is $a_2x^{\frac{m_2}{n}}$. Then $N(f_1)$ has a single side $\Gamma_1$ with the slope $-\frac{d_1}{m_2-m_1}$, and 
    $(d_1,m_2-m_1)=d_2$. This implies that 
    \[F_{\Gamma_1}=(Z^{\frac{d_1}{d_2}}-a_2^{\frac{d_1}{d_2}})^{d_2}\]
   We get the series $f_1$ is of the form 
    \[(y_1^{\frac{d_1}{d_2}}-a_2^{\frac{d_1}{d_2}}x_1^{\frac{m_2-m_1}{d_2}})^{d_2}+\sum_{d_1\alpha_1+(m_2-m_1)\beta_1}c_{\alpha_1,\beta_1}^{(1)}x^{\alpha_1}y^{\beta_1}.\]

    By a very similar argument as in Case 1, the series $f$ is of the form
    \begin{align*}
       \Big{(} (y^{\frac{n}{d_1}}-a^{\frac{n}{d_1}}x^{\frac{m_1}{d_1}})^{\frac{d_1}{d_2}}+H_{2,0} \Big{)}^{d_2}+
       \sum_{n\alpha+m_1\beta+F_1^{(\gamma_2)}\delta_1>d_1F_1^{(\gamma_2)}}c^{(1)}_{\alpha,\beta,\delta_1}x^{\alpha}y^{\beta}(y^{\frac{n}{d_1}}-a^{\frac{n}{d_1}}x^{\frac{m_1}{d_1}})^{\delta_1}
       \end{align*}    
where $(y^{\frac{n}{d_1}}-a^{\frac{n}{d_1}}x^{\frac{m_1}{d_1}})^{\frac{d_1}{d_2}}+H_{2,0}$ has a $y$-root with initial terms $a_1x^{\frac{m_1}{n}}+a_{2}x^{\frac{m_2}{n}}$ and 
with polydromy order $\frac{n}{d_2}$, and $H_{2,0}$ is a linear combination of $x^sy^t$ satisfying $ns+m_1t=\frac{m_1n+d_1(m_2-m_1)}{d_2}=\frac{d_1}{d_2}F_1^{(\gamma_2)}$.
 
By repeating arguments very similar to the ones in Case 1 and Case 2, we conclude that $f$ is of the form 
\[F_1^{d_1}+F^{(1)}_{ext}\]
where $F_{ext}^{(1)}$ is a linear combination of $x^\alpha y^\beta F_1^{\delta_1}$ satisfying 
\[n\alpha+m_1\beta+\big{(}\mathrm{ord}(F_1)\cdot X_{\gamma_2}\big{)}\delta_1> d_1\big{(}\mathrm{ord}(F_1)\cdot X_{\gamma_2}\big{)}.\]
Inductively, by very similar arguments, for any $1\le i\le g-1$, we obtain that $f$ is of the form
\[F_{i,j}^{d_i}+F^{(i,j)}_{ext},\]
where $F^{(i,j)}_{ext}$ is a linear combination of $x^\alpha y^\beta F_1^{\delta_1}\cdots F_{g-1}^{\delta_{g-1}}$ satisfying, 
for $1\le l\le g$, 
\[\mathrm{ord}(x^{\alpha}y^{\beta}F_{1}^{\delta_1}F_{2}^{\delta_2}\cdots F_{g-1}^{\delta_{g-1} })\cdot X_{\gamma_l}>\mathrm{ord}( F_{i,j}^{d_{i}})\cdot X_{\gamma_l}.\]
          
\end{proof}
\begin{definition}\label{D:Fideal}
    Let $S$ be a modified Puiseux series of an irreducible germ $Z$.
    Fixing a set of standard factors $F=\{F_1,\ldots, F_{g-1}\}$ of $S$ (see Definition \ref{D:standardfactors}), for any $c\in \fQ$, 
    we define an ideal $\so_S^{\ge c}$ generated by all terms $x^{p_x}y^{p_0}F_1^{p_1}\cdots F_{g-1}^{p_{g-1}}$\footnote{For the convenience of calculations in 
    Remark \ref{Re:recover}, we denote by $p_0$ the power of $y$.} satisfying the 
    condition 
    \[\rho(x^{p_x}y^{p_0}F_1^{p_1}\cdots F_{g-1}^{p_{g-1}})\ge c,\]
    and an ideal $\so_S^{>c}$ generated by all terms $x^{p_x}y^{p_0}F_1^{p_1}\cdots F_{g-1}^{p_{g-1}}$ satisfying the 
    condition 
    \[\rho(x^{p_x}y^{p_0}F_1^{p_1}\cdots F_{g-1}^{p_{g-1}})> c.\]
    \end{definition}
    \begin{remark}
        For any $c\in\fQ$, the ideal $\so_S^{\ge c}$ (or $\so_S^{> c}$) is independent of the choice of the set of standard factors of $S$. Indeed, let  $F=\{F_1,\ldots, F_{g-1}\}$ and $F'=\{F_1',\ldots, F_{g-1}'\}$
        be two sets of 
        standard factors of $S$. We have
        \[\rho(x^{p_x}y^{p_0}F_1^{p_1}\cdots F_{g-1}^{p_{g-1}})=\rho(x^{p_x}y^{p_0}F_1'^{p_1}\cdots F_{g-1}'^{p_{g-1}})\]
        for any powers $p_x,p_0,p_1,\ldots,p_{g-1}$. Since $F_i$ has a $y$-root $S_i$ with $[S_i]_{<\frac{m_{i+1}}{n}}=[S]_{<\frac{m_{i+1}}{n}}$ and with the characteristic 
        sequence $(\frac{n}{d_i};\frac{m_1}{d_i},\ldots,\frac{m_i}{d_i})$, by Proposition \ref{D:standardfactors}, we can choose $\{F_1,\ldots, F_{i-1}\}$ to be a set of standard 
        factors of $S_i$.
        By Propostion \ref{P:2}, the difference $F_i-F_{i}'$ can be written as a linear combination of $x^\alpha y^{\beta}F_1^{\delta_1}\cdots F_{i-1}^{\delta_{i-1}}$ satisfying, 
        for any $1\le l\le g$, 
        \[\mathrm{ord}(x^\alpha y^{\beta}F_1^{\delta_1}\cdots F_{i-1}^{\delta_{i-1}})\cdot X_{\gamma_l}> \mathrm{ord}(F_{i,0})\cdot X_{\gamma_l}=\mathrm{ord}(F_i)\cdot X_{\gamma_l}\]
        and hence $\rho(F_i-F_i')> \rho(F_i)$. Therefore the ideal $\so_S^{\ge c}$ is well defined for any $c\in\fQ$.
    \end{remark}

\section{The main theorem}
\subsection{Main result}
We can now state the main result of the paper.
Recall that the standard factors $F_1,\ldots, F_{g-1}$ are defined in Definition \ref{D:standardfactors} and, for any $c\in\fQ$, the ideal 
$\so_S^{>c}$ is defined in Definition \ref{D:Fideal}. 

\begin{theorem}\label{T:main2}
    Let $S$ be a modified Puiseux series of an irreducible germ $Z$ which is not tangent to the $y$-axis.
    Let $F=\{F_1,\ldots,F_{g-1}\}$ be a set of standard factors of $S$. Then for $0<\alpha<1$,
   we have $\mathfrak{I}(\alpha Z)=\so_S^{>\alpha}$.
\end{theorem}

\begin{proof}
We start by proving that 
\[G\in \so_S^{\ge \rho(G)},\forall~G\in\fC\{x,y\}.\]
     First, let $G\in\fC\{x,y\}$ be an irreducible series. By Lemma \ref{L:singleconj} and Lemma \ref{L:gap}, it defines a germ of an irreducible plane curve $C_G$. 
     Let 
    \[\pi:Y=Y_k\stackrel{\pi_k}{\xrightarrow{\hspace*{0.7cm}}}
    Y_{k-1}\stackrel{\pi_{k-1}}{\xrightarrow{\hspace*{0.7cm}}}\cdots \xrightarrow{\hspace*{0.7cm}} 
    Y_1\stackrel{\pi_1}{\xrightarrow{\hspace*{0.7cm}}} Y_0=\fC^2\] 
    be the standard resolution of $Z$ and $q_1,\ldots,q_k$ be the centers of the blow-ups $\pi_1,\ldots,\pi_k$, and similarly, let
    \[\pi':Y'=Y'_{k'}\stackrel{\pi'_{k'}}{\xrightarrow{\hspace*{0.7cm}}}
    Y'_{k'-1}\stackrel{\pi_{k'-1}}{\xrightarrow{\hspace*{0.7cm}}}\cdots \xrightarrow{\hspace*{0.7cm}} 
    Y'_1\stackrel{\pi'_1}{\xrightarrow{\hspace*{0.7cm}}} Y_0=\fC^2\] 
    be the standard resolution of $C_G$ and $q'_1,\ldots,q'_{k'}$ be the centers of the blow-ups $\pi'_1,\ldots,\pi'_{k'}$.
      If $C_G$ is tangent 
      to the $y$-axis, then we reverse the order of coordinates and find a $x$-root $S'$ of $G$, which is a fractional power series of $y$ with polydromy order $m'$ and of the form
      \[c_{m'}y+\cdots+c_{hm'}y^h+a'y^{\frac{n'}{m'}}+\cdots\]
      where $m'<n'$ and $h=[\frac{n'}{m'}]$. Applying Proposition \ref{P:2} to $S'$ with coordinates $y,x$, and set $\textrm{gcd}(n',m')=d'$,
      we obtain that $G$ is of the form
      \[\Big{(}(x-c_{m'}y+\cdots+c_{hm'}y^h)^{\frac{m'}{d'}}-(a')^{\frac{m'}{d'}}y^{\frac{n'}{d'}}\Big{)}^{d'}
      +\sum_{n'\alpha+m'\beta>n'm'}c_{\alpha,\beta}y^\alpha (x-c_{m'}y+\cdots+c_{hm'}y^h)^\beta.\]   
      So $G\in (x,y)^{m'}$.  We know for any $x^sy^t\in (x,y)^{m'}$, 
      $\mathrm{ord}(x^sy^t)\cdot X_{\gamma_i}=\frac{sn+m_1t}{d_i}\ge\frac{m'n}{d_i}=\mathrm{ord}(x^{m'})\cdot X_{\gamma_i}$.
       Hence $G\in \so_F^{\ge\rho(x^{m'})}$.
      On the other hand, we see that $\textrm{ord}(G)$ is of the form $(m',0,\ldots,0)$ (since after the first blow-up the strict transform of $G$ will never pass the 
      center $q_2$). 
     By Lemma \ref{L:7}, we get $\textrm{ord}(G)\cdot X_{\gamma_i}=\frac{m'n}{d_i}$. Then 
      by Lemma \ref{L:7} and  (\ref{E:8}), we have $\rho(G)=\rho(x^{m'})$.  Hence we obtain $G\in \so_S^{\ge \rho(G)}$.
Now we suppose that $G$ is not tangent to $y$-axis. Denote by $(n';m_1',\ldots,m_{g'}')$ the characteristic sequence of $S'$ 
and denote by
\begin{align}\label{E:Zprimemulti}
    (M_1',\ldots,M_{k'}')=(n',\ldots,n',r_{1,1}',\ldots,r'_{g',k_{g'}})
\end{align}
the multiplicity sequence of $S'$,
where $n'$ appears $h_{1,0}'$ times, $r_{1,1}'$ appears $h_{1,1}'$ times, and so on, where $h'_{\bullet,\bullet}$ and $r'_{\bullet,\bullet}$ are invariants provided by the chain of 
$g'$ Euclidean algorithm similar to (\ref{E:3}).
Similar to Notation \ref{N:gamtau}, let $\gamma_0'=0$ and for $1\le i\le g'$, let
    \[\gamma_i'=\gamma_{i-1}'+\sum_{j=0}^{k_i}h'_{i,j},\hspace{10pt}\textrm{and}\hspace{10pt}\tau'_{i-1}=\gamma'_{i-1}+h'_{i,0}+1.\]
       We claim that, $G$ can be written as a linear combination of $x^\alpha y^\beta F_1^{\delta_1}\cdots F_{g-1}^{\delta_{g-1}}$ 
    satisfying 
    \[ \mathrm{ord}(x^\alpha y^\beta F_1^{\delta_1}\cdots F_{g-1}^{\delta_{g-1}})\cdot X_{\gamma_i}\ge \mathrm{ord}(G)\cdot X_{\gamma_i}.\]

    Suppose first that $q_1=q_1',\ldots,q_j=q_j'$ and $q_{j+1}\neq q_{j+1}'$ with $j<\min\{k,k'\}$.
 Suppose further that $q_j=q_j'$ is a free point of $Z$ and $Z'$, and suppose that $q_{j+1}\neq q_{j+1}'$ are free points of $Z$ and $Z'$, respectively.
Then there exists $0\le i\le g-1$ such that $\gamma_i=\gamma_i'<j< \min\{\tau_i,\tau_i'\}$. Let $r=j-\gamma_i$. By Proposition \ref{P:3}, we know that $G$ has 
a $y$-root $S'$ which is of the form 
\[S'=[S]_{\le\frac{m_i+rd_i}{n}}+bx^{\frac{m_i'+rd_i'}{n'}}+\cdots\]
where $b\neq 0$. By Proposition \ref{P:2}, we find $G$ is of the form
\[(F_{i-1}^{\frac{d_{i-1}'}{d_i'}}+H_{i,0}+\cdots+H_{i,r}+cH_{i,r})^{d_i'}+F_{ext}^{(i)}\]
where $F_{ext}^{(i)}$ is a linear combination of $x^\alpha y^\beta F_1^{\delta_1}\cdots F_{i-1}^{\delta_{i-1}}$ satisfying 
\[\mathrm{ord}(x^\alpha y^\beta F_1^{\delta_1}\cdots F_{i-1}^{\delta_{i-1}})\cdot X_{\gamma_i}> \mathrm{ord}(G)\cdot X_{\gamma_i}.\]
On the other hand, we have 
\[\mathrm{ord}(G)=[\overbrace{n',\ldots,n',\ldots,d_i',\ldots,d_i'}^{\gamma_i'},\overbrace{d_i',\ldots,d_i'}^r,0,\ldots,0].\]
So by Lemma \ref{L:7}, we get 
\[\textrm{ord}(G)\cdot X_{\gamma_l}=\begin{cases}
    d_i'\cdot \frac{M_1^2+\cdots+M_{\gamma_l}^2}{d_i d_l}&~l\le i\\
    d_i'\cdot \frac{M_1^2+\cdots+M_j^2}{d_i d_l}&~l>i.\\
\end{cases}
\]
By Proposition {D:standardfactors}, we know $F_{i,r}$ has a $y$-root has a partial sum equal to $[S]_{\le \frac{m_i+rd_i}{n}}$ and with the polydromy order $\frac{n}{d_i}$. By Proposition 2.20, we then know $F_{i,r}$ goes through at least $q_1,\ldots,q_j$. So we get $\mathrm{ord}(F_{i,r})\cdot X_{\gamma_l}\ge\frac{1}{d_i'}\mathrm{ord}(G)\cdot X_{\gamma_l}$. Therefore the claim is true in this case. 
For the other cases, following a very similar argument, we proved the claim is true. This implies that $G\in\so_S^{\ge \rho(G)}$. 

When $G$ is reducible, we can write $G=G_1^{r_1}\cdots G_m^{r_m}$ as a decomposition of irreducible factors.
Without loss of generality, assume $m=2$. By the argument above we have two inequalities:
for $1\le l\le g$, and $i=1,2$, $G_i$ is a linear combination of $x^{\alpha_i}y^{\beta_i}F_1^{\delta_{i,1}}\cdots F_{g-1}^{\delta_{i,g-1}}$ satisfying 
\[\alpha_i\frac{n}{d_l}+\beta_i\frac{m_1}{d_l}+
 \sum_{w=1}^{g-1}\big{(}\mathrm{ord}(F_w)\cdot X_{\gamma_l}\big{)}\delta_{i,w}\ge \textrm{ord}(G_i)\cdot X_{\gamma_l}.\]
So $G_1\cdot G_2\in \so_S^{\ge\rho(G_1G_2)}$ and hence $G\in \so_S^{\ge \rho(G)}$.

   Following Definition \ref{D:8} and by (\ref{E:multipl}), we know for $0\le \alpha<1$,
   \[G\in \mathfrak{I}(\alpha Z)\Longleftrightarrow \rho(G)> \alpha.\]
 Therefore $\mathfrak{I}(\alpha Z)\subset \so_S^{>\alpha}$. Conversely, by Definition \ref{D:8} and (\ref{E:multipl}), we know that generators of 
   $\so_S^{>\alpha}$ have $\rho$-value greater than $\alpha$ and hence are in the ideal $\mathfrak{I}(\alpha Z)$. So $\so_S^{>\alpha}\subset \mathfrak{I}(\alpha Z)$.
   Therefore we obtain $\mathfrak{I}(\alpha Z)=\so_S^{>\alpha}$
 \end{proof}
 \begin{corollary}\label{C:jumpingnumber}
    Let $S$ be a modified Puiseux series of an irreducible germ $Z$ which is not tangent to the $y$-axis.
    Let $F=\{F_1,\ldots,F_{g-1}\}$ be a set of standard factors of $S$. 
     The set of jumping numbers of $Z$ between 0 and 1 is
     \[\{\rho(x^{p_x}y^{p_0}F_1^{p_1}\cdots F_{g-1}^{p_{g-1}})~|~\forall~p_x,p_0,p_1,\ldots,p_{g-1}\in\rN\}\cap (0,1).\]
 \end{corollary}
 \begin{remark}\label{Re:recover}
     Let $(M_1,\ldots,M_k)$ be the multiplicity sequence of $Z$. Set $d_0=n$ and
     \[B_\nu=\frac{M_1^2+\cdots+M_{\gamma_v}^2}{d_{\nu-1}}\]
    so that $B_\nu$ is the same as the notation $b_\nu$ in \cite{Jar06}. By Lemma \ref{L:calculations}, for any $p_x,p_0,p_1,\ldots,p_{g-1}\in\rN$, we denote by 
     \[\rho(x^{p_x}y^{p_0}F_1^{p_1}\cdots F_{g-1}^{p_{g-1}})=\min_{1\le i\le g}\Omega_i\]
     where, for $1\le l\le g$, 
     \[\Omega_l=\frac{m_l+n+np_x+\sum_{j=0}^{l-1}B_jp_j}{d_{l-1}B_l}+\sum_{j=l}^{g-1}\frac{p_j}{d_j}.\]
     Notice that \[M_1^2+\cdots+M_{\gamma_l}^2=m_1n+d_1(m_2-m_1)+d_2(m_3-m_2)+\cdots +d_{l-1}(m_l-m_{l-1}).\]
     By elementary calculations, we can write 
     \[\Omega_l=\frac{p_{l-1}+1}{d_{l-1}}+\frac{t_l+1}{B_l}+\sum_{j=l}^{g-1}\frac{p_j}{d_j} \]
     with $t_l\in \rN$, and we obtain
     \[\Omega_{l}\le\Omega_{l+1}\Longleftrightarrow \frac{p_{l-1}+1}{d_{l-1}}+\frac{t_l+1}{B_l}\le \frac{1}{d_l}. \]
     So the above corollary recovers formulas given in 
      \cite[Theorem 9.4]{Jar06}. 
 \end{remark}
 \begin{example}\label{Ex:1}
    Consider the germ 
    \[Z=(y^4-4x^2y^3+4x^4y^2-2x^3y^2+4x^5y-4x^6y+x^6=0).\] 
    After running the Newton-Puiseux Algorithm for $Z$, we know it is irreducible since it has a single conjugacy class of Puiseux series. The characteristic sequence of $Z$ is (4;6,9).
  Observe that $Z$ has a Puiseux series $S$ with a partial sum 
\[x^\frac{3}{2}+x^2+x^{\frac{9}{4}}\]
 and choose a set of standard factors of $S$ 
    \[F=\{F_1=y^2-x^3-2x^2y\}.\]
 By Theorem \ref{T:main2}, a set of  generators of 
 $\mathfrak{I}(\alpha Z)$ are all the polynomials of the form $x^{p_x}y^{p_0}F_1^{p_1}$ satisfying
  \[\min\Big{\{}\frac{5+2p_x+3p_0+6p_1}{12},\frac{13+4p_x+6p_0+15p_1}{30}\Big{\}}>\alpha.\]
Then we can describe the multiplier ideals $\mathfrak{I}(\alpha Z)$ with $0<\alpha<1$ explicitly as in Table \ref{tab:table2} below.
\end{example}

\subsection{A question}

 It is well known that two irreducible plane curves are topologically equivalent if and only if they have the same characteristic sequence (see for example \cite[Theorem 21]{BK86}). 
J\"{a}rvilehto \cite{Jar06} proved that the data of jumping numbers of multiplier ideals between 0 and 1 of the irreducible plane curve 
is the same as the data of the characteristic sequence, and hence determines the topological equivalence class.

We notice that $I_0(\alpha Z)= I_0(\alpha Z')$ for all $0<\alpha\le 1$ after a possible holomorphic change of coordinates 
    is not a sufficient condition for the analytic equivalence of $Z$ and $Z'$. 
     For example, set $Z=(y^5-x^6=0)$ and $Z'=(y^5-x^6-5x^4y^2=0)$. Then by Theorem \ref{T:main2}, $I_0(\alpha Z)=I_0(\alpha Z')$ for any $0<\alpha\le 1$. 
     (Note that $I_0(\alpha Z)=\mathfrak{I}\big{(}(\alpha-\epsilon)Z\big{)}=\so_S^{\ge\alpha}$, where $0<\epsilon\ll 1$.) However, $Z$ and $Z'$ are not analytically isomorphic (see \cite[Chapter V \S 4]{Zar86}).

In \cite{MP16}, \cite{MP18a} and \cite{MP18b}, the authors defined Hodge ideals, 
which contain richer information about the singularity than multiplier ideals. 
It is natural to ask if they can determine a more subtle 
equisingularity equivalence class. Popa  asked the following:
\begin{question}\cite{Pop19}
    Assume that $Z$ and $Z'$ are two germs of irreducible plane curve singularities with the same characteristic sequence.
   Are $Z$ and $Z'$ analytically equivalent if and only if $I_p(\alpha Z)=I_p(\alpha Z')$ for all $p\ge 0$ and 
all $0<\alpha\le 1$, after a possible holomorphic change of coordinates? Is it in fact enough to consider only $p=0,1$?
\end{question}

\begin{table}[h!]
    \begin{center}
      \caption{Multiplier ideals for Example \ref{Ex:1}}
      \label{tab:table2}
      \begin{tabular}{l|r} 
        \textbf{jumping number} & \textbf{multiplier ideal}\\
        $\xi_i$ &  $\mathfrak{I}(\alpha Z),\xi_{i-1}\le \alpha<\xi_i$\\[1ex]
        \hline
        \\[-2ex]
        $\frac{5}{12}$& $\fC\{x,y\}$\\[1ex]
        \hline
        \\[-2ex]
        $\frac{17}{30}$ & $(x,y)$\\[1ex]
        \hline
        \\[-2ex]
        $\frac{19}{30}$ & $(x^2,y)$\\[1ex]
        \hline
        \\[-2ex]
        $\frac{21}{30}$ & $(x,y)^2$\\[1ex]
        \hline
        \\[-2ex]
        $\frac{23}{30}$ & $(x^3,xy,y^2)$\\[1ex]
        \hline
        \\[-2ex]
        $\frac{25}{30}$ & $(x^3,x^2y,y^2)$\\[1ex]
        \hline
        \\[-2ex]
        $\frac{27}{30}$ & $(y^2-x^3,x^2y,x^4,xy^2,y^3)$\\[1ex]
        \hline
        \\[-2ex]
        $\frac{11}{12}$ & $(y^2-x^3-2x^2y,x^4,x^3y, xy^2,y^3)$\\[1ex]
        \hline
        \\[-2ex]
        $\frac{29}{30}$ & $(x^3y,x^4,xy^2,y^3)$\\[1ex]
        \hline
        \\[-2ex]
        $~1$ & $(xy^2-x^4,x^5,x^3y, x^2y^2,y^3)$\\[1ex]
        \hline
      \end{tabular}
    \end{center}
  \end{table}


\begin{thebibliography}{ZZZZ99}
\bibitem[AM73a]{AM73a} Shreeram S. Abhyankar and Tzuong-tsieng Moh, \textit{Newton-Puiseux expansion and generalized Tschirnhausen transformation. I}, I. J. Reine Angew. Math., 260: 47-83, 1973.

\bibitem[AM73b]{AM73b} Shreeram S. Abhyankar and Tzuong-tsieng Moh, \textit{Newton-Puiseux expansion and generalized Tschirnhausen transformation. II},  I. J. Reine Angew. Math., 261: 29-54, 1973.
\bibitem[AMB17]{AAB17} Maria Alberich-Carrami\~{n}ana, Josep \`{A}lvarez Montaner and Gulliem Blanco, \textit{Monomial generators of complete planar ideals}, preprint, arXiv:1701.03503v3 (2017).
\bibitem[AMD16]{AAD16} Maria Alberich-Carrami\~{n}ana, Josep \`{A}lvarez Montaner and Ferran Dachs-Cadefau, \textit{Multiplier ideals in two-dimensional local rings with rational singularities}, Michigan Math. J. \textbf{65} (2016), no. 2, 287-320.
 
    \bibitem[BK86]{BK86} Egbert Brieskorn and Horst Kn\"{o}rrer, \textit{Plane algebraic curves}, translated from the German original by John Stillwell,
     Modern Birkh\"{a}user Classics, Birkh\"{a}user/Springer Basel AG, Basel,1986.
     \bibitem[Bli04]{Bli04} Manuel Blickle, \textit{Multiplier ideals and modules on toric varieties},  Math. Z. \textbf{248} (2004), no. 1, 113-121.
\bibitem[Cas00]{Cas00} Eduardo Casas-Alvero, \textit{Singularities of plane curves}, London Math. Soc. Lecture Note Series, Vol. \textbf{276},
 Cambridge University Press, 2000.
 \bibitem[Dur18]{Dur18} Carlos Rodrigo Guzm\'{a}n Dur\'{a}n, \textit{Ideales de Multiplicadores de Curvas planas irreducibles}, Thesis (Ph.D.)-Centro de Investigaci\'{o}n en Matem\'{a}ticas, 2018.
 \bibitem[EC24]{EC24} Federigo Enriques and Oscar Chisini, \textit{Lezioni sulla teoria geometrica delle equazioni e delle funzioni algebriche},  Vol. I, II, III, [Lessons on the geometric theory of equations and algebraic functions], Reprint of the 1924 and 1934 editions. Collana di Matematica [Mathematics Collection], \textbf{5},  Zanichelli Editore S.p.A., Bologna, 1985.
    \bibitem[GHM16]{GHM16} Carlos Galindo and Fernando Hernando and Francisco Monserrat, \textit{The log-canonical threshold of a plane curve},  
    Math. Proc. Cambridge Philos. Soc. \textbf{160} (2016), no. 3, 513-535. 
    \bibitem[How11]{How11} Jason Andrew Howald, \textit{Multiplier ideals of monomial ideals}, Trans. Amer. Math. Soc. 
     \textbf{353} (2001), no. 7, 2665-2671.
    \bibitem[Igu77]{Igu77} Jun-ichi Igusa, \textit{On the first terms of certain asymptotic expansions}, Complex analysis and algebraic geometry, 357-368, Iwanami Shoten, Tokyo, 1977.
\bibitem[HJ18]{HJ18}  Eero Hyry, Tarmo J\"{a}rvilehto,  \textit{A formula for jumping numbers in a two-dimensional regular local ring}, 
 J. Algebra 516 (2018), 437�-470.
   \bibitem[J\"{a}r07]{Jar07} Tarmo J\"{a}rvilehto, \textit{Jumping numbers of a simple complete ideal in a two-dimensional regular local ring}, Thesis (Ph.D.)-University of Helsinki, 2007.
    \bibitem[J\"{a}r11]{Jar06} Tarmo J\"{a}rvilehto, \textit{Jumping numbers of a simple complete ideal in a two-dimensional regular local ring}, Mem. Amer. Math. Soc. 214 (2011), no. 1009, viii+78 pp.
      \bibitem[Kuw99]{Kuw99} Takayasu Kuwata, \textit{On log canonical thresholds of reducible plane curves}, Amer. J. Math. \textbf{121} (1999), no. 4, 701-721.
     \bibitem[Laz04]{Laz04} Robert Lazarsfeld, \textit{Positivity in algebraic geometry II: Positivity for vector bundles, and multiplier ideals},  
     Ergebnisse der Mathematik und ihrer Grenzgebiete, 3, Folge,  A Series of Modern Surveys in Mathematics [Results in Mathematics and Related Areas, 3rd Series, A Series of Modern Surveys in Mathematics], vol. 49, Springer-Verlag, Berlin, 2004.
    \bibitem[Lip94]{Lip94} Joseph Lipman, \textit{Proximity inequalities for complete ideals in two-dimensional regular local rings}, Commutative algebra: syzygies, multiplicities, and birational algebra (South Hadley, MA, 1992), 293-306, Contemp. Math., \textbf{159}, Amer. Math. Soc., Providence, RI, 1994.
    \bibitem[MP16]{MP16} Mircea Musta\c{t}\u{a} and Mihnea Popa, \textit{Hodge ideals},
preprint arXiv:1605.08088, to appear in Memoirs of the AMS(2016).
    \bibitem[MP18a]{MP18a} Mircea Musta\c{t}\u{a} and Mihnea Popa, \textit{Hodge ideals for $\mathbb{Q}$-divisors:~birational approach}, preprint arXiv:   1807.01932v2 (2018). 
    \bibitem[MP18b]{MP18b} Mircea Musta\c{t}\u{a} and Mihnea Popa, \textit{Hodge ideals for $\mathbb{Q}$-divisors, V-filtration, and minimal exponent}, 
    preprint arXiv:1807.01935v3 (2018).
    \bibitem[Mus06]{Mus06} Mircea Musta\c{t}\u{a}, \textit{Multiplier ideals of hyperplane arrangements}, Trans. Amer. Math. Soc. \textbf{358} (2006), no. 11, 5015-5023.
    \bibitem[Nai09]{Nai09} Daniel Naie, \textit{Jumping numbers of a unibranch curve on a smooth surface}, Manuscripta
    Math. \textbf{128} (2009), no. 1, 33-49.
    \bibitem[Pop19]{Pop19} Mihnea Popa, \textit{Personal communication}, January, 2019.
    \bibitem[Sai00]{Sai00} Morihiko Saito, \textit{Exponents of an irreducible plane curve singularity}, preprint, arXiv:0009133 (2000).
    \bibitem[ST06]{ST06} Karen E. Smith and Howard M. Thompson, \textit{Irrelevant exceptional divisors for curves on a smooth surface}, Algebra, geometry and their interactions, 245-254, 
    Contemp. Math. \textbf{448}, Amer. Math. Soc., Providence, RI, 2007.  
    \bibitem[Tei07]{Tei07} Zachariah C. Teitler, \textit{Multiplier ideals of general line arrangements in}~$\fC^3$, Comm. Algebra \textbf{35} (2007), no. 6, 1902-1913.
    \bibitem[Tei08]{Tei08} Zachariah C. Teitler, \textit{A note on Musta\c{t}\u{a}'s computation of multiplier ideals of hyperplane arrangements}, Proc. Amer. Math. Soc. \textbf{136} (2008), no. 5, 1575-1579.
    \bibitem[Tho14]{Tho14} Howard M. Thompson, \textit{Multiplier ideals of monomial space curves}, Proc. Amer. Math. Soc. Ser. B \textbf{1} (2014), 33-41.
    \bibitem[Tho16]{Tho16} Howard M. Thompson, \textit{A short note on the multiplier ideals of monomial space curves}, J. Pure Appl. Algebra \textbf{220} (2016), no. 6, 2459-2466.
    \bibitem[Tuc10a]{Tuc10a} Kevin Tucker, \textit{Jumping numbers and multiplier ideals on algebraic surfaces}, Thesis (Ph.D.)-University of Michigan, 2010,114pp.
    \bibitem[Tuc10b]{Tuc10} Kevin Tucker, \textit{Jumping numbers on algebraic surfaces with rational singularities}, Trans. Amer. 
     Math. Soc. \textbf{362} (2010), no. 6, 3223-3241. 
  \bibitem[VD19]{VD19} Manuel Gonz\'{a}lez Villa and Carlos Rodrigo Guzm\'{a}n Dur\'{a}n, \textit{Multiplier ideals of plane branches}, unpublished manuscript, 2019.
     \bibitem[ZS60]{ZS60} Oscar Zariski and Pierre Samuel, \textit{Commutative algebra}, Vol. II, The University Series in Higher Mathematics. D. Van Nostrand Co., Inc., Princeton, N. J.-Toronto-London-New York, 1960.
     \bibitem[Zar86]{Zar86} Oscar Zariski, \textit{The moduli problem for plane branches}, (with an appendix by Bernard Teissier),
      translated by Ben Lichtin from the 1973 French original, University lecture series, Vol. \textbf{39}, Providence, RI, 2006.
      
       \end{thebibliography}
\end{document}